\documentclass[a4paper]{amsart}
\usepackage{amscd}
\usepackage{amsmath}
\usepackage{amssymb}
\usepackage{amsthm}
\usepackage{bbm}
\usepackage{stmaryrd}

\usepackage[T1]{fontenc}

\newif\ifpdf
\ifx\pdfoutput\undefined
   \pdffalse        
\else
   \pdfoutput=1     
   \pdftrue
\fi

\ifpdf
   \usepackage[pdftex]{graphicx}
\else
   \usepackage{graphicx}
\fi

\numberwithin{equation}{section} \swapnumbers

\newtheorem{satz}{Satz}[section]

\newtheorem{theorem}[satz]{Theorem}
\newtheorem{proposition}[satz]{Proposition}
\newtheorem{corollary}[satz]{Corollary}
\newtheorem{lemma}[satz]{Lemma}

\newtheorem{definition}[satz]{Definition}

\newtheorem{remark}[satz]{Remark}

\newtheorem{example}[satz]{Example}

\newcommand{\bbr}{\mathbb{R}}
\newcommand{\bbc}{\mathbb{C}}

\newcommand{\bbn}{\mathbb{N}}
\newcommand{\bbp}{\mathbb{P}}

\newcommand{\bbm}{\mathbb{M}}

\newcommand{\calb}{\mathcal{B}}
\newcommand{\cald}{\mathcal{D}}
\newcommand{\calf}{\mathcal{F}}
\newcommand{\calm}{\mathcal{M}}

\begin{document}

\title[Affine realizations for stochastic partial differential equations]{Existence of affine realizations for stochastic partial differential equations driven by L\'{e}vy processes}
\author{Stefan Tappe}
\address{Leibniz Universit\"{a}t Hannover, Institut f\"{u}r Mathematische Stochastik, Welfengarten 1, 30167 Hannover, Germany}
\email{tappe@stochastik.uni-hannover.de}
\thanks{I am grateful to Ozan Akdogan, Stefan Weber and two anonymous referees for valuable comments and suggestions.}
\begin{abstract}
The goal of this paper is to clarify when a semilinear stochastic partial differential equation driven by L\'{e}vy processes admits an affine realization. Our results are accompanied by several examples arising in natural sciences and economics.
\end{abstract}
\maketitle

\section{Introduction}

The goal of this paper is to clarify when a semilinear stochastic partial differential equation (SPDE) of the form
\begin{align}\label{SPDE}
\left\{
\begin{array}{rcl}
dr_t & = & (A r_t + \alpha(r_t))dt +
\sigma(r_{t-})dX_t
\medskip
\\ r_0 & = & h_0
\end{array}
\right.
\end{align}
in the spirit of \cite{P-Z-book} driven by a $\bbr^m$-valued L\'{e}vy process $X$ (for some positive integer $m \in \bbn$) admits an affine realization. Affine realizations are particular types of finite dimensional realizations (FDRs). Denoting by $H$ the state space of (\ref{SPDE}), which we assume to be a separable Hilbert space, the idea of a FDR is that for each starting point $h_0 \in H$ we can express the weak solution $r$ to (\ref{SPDE}) as
\begin{align}\label{r-FDR}
r = \varphi(Y)
\end{align}
for some $\bbr^d$-valued process $Y$ (where $d \in \bbn$ is a positive integer) and a deterministic mapping $\varphi : \bbr^d \to H$, which makes the infinite dimensional SPDE (\ref{SPDE}) more tractable. If we have a representation of the form (\ref{r-FDR}), then the mapping $\varphi$ is the parametrization of an invariant submanifold $\calm$.

We speak about an affine realization if for each starting point $h_0 \in H$ we can express the weak solution $r$ to (\ref{SPDE}) as
\begin{align}\label{decomp-intro}
r = \psi + Y
\end{align}
with a deterministic curve $\psi : \bbr_+ \to H$ and a stochastic process $Y$ having values in a finite dimensional subspace $V \subset H$. In this case, we also say that the SPDE (\ref{SPDE}) has an affine realization generated by $V$, and the invariant manifold $(\calm_t)_{t \in \bbr_+}$ is a collection of affine spaces $\calm_t = \psi(t) + V$, also called a foliation.

Note that the existence of an affine realization makes the infinite dimensional SPDE (\ref{SPDE}) very tractable, because then we have a simple structure of the invariant manifolds, which might be more complicated for a general FDR. Surprisingly, in many cases we can deduce the existence of an affine realization from the existence of a FDR:
\begin{itemize}
\item As shown in \cite{Filipovic}, the existence of a FDR for the Wiener process driven HJMM equation implies the existence of an affine realization. Here we use the name HJMM equation, as it is the Heath-Jarrow-Morton (HJM) model from \cite{HJM} with Musiela parametrization presented in \cite{Musiela}.

\item As shown in \cite{Tappe-rank}, for the general L\'{e}vy process driven SPDE (\ref{SPDE}) the flatness of an invariant manifold is at least equal to the number of driving sources with small jumps. Thus, if the SPDE (\ref{SPDE}) has driving L\'{e}vy processes with small jumps, then every FDR up to a certain dimension must be an affine realization.
\end{itemize}
There is a substantial literature about invariant manifolds and FDRs for SPDEs. Stochastic invariance of a given finite dimensional submanifold has been studied in \cite{Filipovic-inv}, and -- based on the support theorem presented in \cite{Nakayama-Support} -- in \cite{Nakayama} for SPDEs driven by Wiener processes, in \cite{FTT-manifolds} for SPDEs driven by Wiener processes and Poisson random measures, and in \cite{Tappe-rank} for SPDEs driven by L\'{e}vy processes. The existence of FDRs for the HJMM equation driven by Wiener processes has intensively been studied in the literature, and we refer to \cite{Bj_Sv, Bj_La, Filipovic, Filipovic-Teichmann-royal} and references therein, and to \cite{Bjoerk} for a survey. Furthermore, the existence of affine realizations for the HJMM equation has been studied in \cite{Tappe-Wiener, Tappe-affin} with a driving Wiener process, and in \cite{Tappe-Levy, Platen-Tappe} with a driving L\'{e}vy process.

The goal of this paper is to clarify when the general SPDE (\ref{SPDE}) driven by L\'{e}vy processes has an affine realization, which has not been treated in the literature so far. Compared to the aforementioned papers \cite{Tappe-Wiener, Tappe-affin, Tappe-Levy, Platen-Tappe}, we use a slightly different concept of an affine realization:
\begin{itemize}
\item We demand that for every starting point $h_0 \in H$ the weak solution $r$ to (\ref{SPDE}) is of the form (\ref{decomp-intro}), whereas in the aforementioned papers this is only demanded for every $h_0 \in \cald(A)$, which denotes the domain of the linear operator $A : \cald(A) \subset H \to H$ appearing in (\ref{SPDE}). 

\item On the other hand, our definition is more relaxed, because we only demand that the invariant foliations are $C^0$-foliations, whereas in the aforementioned papers they have to be $C^1$-foliations.
\end{itemize}
Now, let us outline the main results of this paper. Concerning the precise assumptions on the L\'{e}vy process $X$ and the parameters $(A,\alpha,\sigma)$ of the SPDE (\ref{SPDE}) we refer to the beginning of Section~\ref{sec-SPDEs}. We fix a finite dimensional subspace $V \subset H$ and agree on the following terminology. We say that the subspace $V$ is
\begin{itemize}
\item \emph{$A$-semi-invariant} if $A(V \cap \cald(A)) \subset V$;

\item \emph{$A$-invariant} if $V \subset \cald(A)$ and $A(V) \subset V$.
\end{itemize}
Our first main result presents necessary and sufficient conditions for the existence of an affine realization generated by $V$ in terms of the parameters $(A,\alpha,\sigma)$ of the SPDE (\ref{SPDE}). We will provide the proof in Section \ref{sec-realizations}.

\begin{theorem}\label{thm-intro}
Suppose that the subspace $V$ is $A$-semi-invariant. Then the SPDE (\ref{SPDE}) has an affine realization generated by $V$ if and only if the following three conditions are fulfilled:
\begin{enumerate}
\item $V$ is $A$-invariant (or equivalently: $V \subset \cald(A)$).

\item For each $h \in H$ the projection $\Pi_{(\bullet,V)} \alpha$ is constant on $h+V$.

\item $\sigma^k(H) \subset V$ for all $k=1,\ldots,m$.
\end{enumerate}
\end{theorem}

Concerning Theorem~\ref{thm-intro}, let us remark the following two points:
\begin{itemize}
\item The assumption that the subspace $V$ is $A$-semi-invariant does not mean a restriction. Indeed, we will show that we can always rewrite the SPDE (\ref{SPDE}) equivalently as
\begin{align}\label{SPDE-B}
\left\{
\begin{array}{rcl}
dr_t & = & (B r_t + \beta(r_t))dt +
\sigma(r_{t-})dX_t
\medskip
\\ r_0 & = & h_0,
\end{array}
\right.
\end{align}
such that the subspace $V$ is $B$-semi-invariant; see Lemmas~\ref{lemma-transformation} and \ref{lemma-equiv} below.

\item In condition (2), we denote by $\Pi_{(\bullet,V)} \alpha$ the projection of the drift $\alpha$ on the first coordinate $U$ according to some direct sum decomposition $H = U \oplus V$ of the Hilbert space. Condition (2) does not depend on the choice of the subspace $U$ appearing in $H = U \oplus V$, which follows from Lemma~\ref{lemma-const} below.
\end{itemize}
Theorem~\ref{thm-intro} has the following immediate consequence:

\begin{corollary}\label{cor-intro}
Suppose that the following three conditions are fulfilled:
\begin{enumerate}
\item $V$ is $A$-invariant.

\item $\alpha(H) \subset V$.

\item $\sigma^k(H) \subset V$ for all $k=1,\ldots,m$.
\end{enumerate}
Then the SPDE (\ref{SPDE}) has an affine realization generated by $V$.
\end{corollary}

In applications, one is often interested in linear SPDEs of the type (\ref{SPDE}), which means that the drift $\alpha$ appearing in (\ref{SPDE}) is constant. We will see that for linear SPDEs we can even skip the assumption that the subspace $V$ is $A$-semi-invariant, and obtain our second main result, which we will also prove in Section \ref{sec-realizations}.

\begin{theorem}\label{thm-linear}
Suppose that the SPDE (\ref{SPDE}) is linear. Then it has an affine realization generated by $V$ if and only if the following two conditions are fulfilled:
\begin{enumerate}
\item $V$ is $A$-invariant.

\item $\sigma^k(H) \subset V$ for all $k=1,\ldots,m$.
\end{enumerate}
\end{theorem}

So far, we have specified a finite dimensional subspace $V$ in advance, and asked for an affine realization generated by $V$. If the SPDE (\ref{SPDE}) is linear, then there are two approaches in order to analyze the existence of an affine realization without specifying a subspace in advance:
\begin{itemize}
\item We will present a result (see Theorem~\ref{thm-qe} below) which states that the linear SPDE (\ref{SPDE}) has an affine realization if and only if the volatility is quasi-exponential.

\item Another approach is to determine all finite dimensional $A$-invariant subspaces, and to apply Theorem~\ref{thm-linear}. This leads to a generalized eigenvalue problem, which we will illustrate in Section \ref{sec-examples} by means of several examples.
\end{itemize}
The remainder of this paper is organized as follows. In Section \ref{sec-SPDEs} we provide the required preliminaries about SPDEs driven by L\'{e}vy processes, and in Section \ref{sec-Hilbert} we provide the required results about direct sum decompositions of Hilbert spaces. In Section \ref{sec-foliations} we present our results about $C^0$-foliations, and in Section \ref{sec-realizations} we provide the proofs of our main results concerning the existence of affine realizations. In Section \ref{sec-HJMM}, we study the HJMM equation as an example of a nonlinear SPDE, and in Section \ref{sec-examples} we present several examples of linear SPDEs arising in natural sciences and economics.

\section{SPDEs driven by L\'{e}vy processes}\label{sec-SPDEs}

In this section, we provide the required preliminaries about SPDEs driven by L\'{e}vy processes. Let $(\Omega,\calf,(\calf_t)_{t \in \bbr_+},\bbp)$ be a filtered probability space satisfying the usual conditions. Let $X$ be a $\bbr^m$-valued L\'{e}vy process for some positive integer $m \in \bbn$ such that its components $X^1,\ldots,X^m$ are nontrivial square-integrable martingales. Let $H$ be a separable Hilbert space and let $A : \cald(A) \subset H \to H$ be the infinitesimal generator of a $C_0$-semigroup on $H$. We assume that the generated semigroup $(S_t)_{t \geq 0}$ is pseudo-contractive; that is, there exists a constant $\beta \in \bbr$ such that
\begin{align*}
\| S_t \| \leq e^{\beta t} \quad \text{for all $t \geq 0$.}
\end{align*}
Furthermore, let $\alpha : H \to H$ and $\sigma : H \to H^m$ be Lipschitz continuous mappings. 

\begin{remark}
Under the above conditions, for each $h_0 \in H$ the SPDE (\ref{SPDE}) has a unique weak solution; that is, a $H$-valued c\`{a}dl\`{a}g adapted process $r$, unique up to indistinguishability, such that for each $\xi \in \cald(A^*)$ we have
\begin{align*}
\langle \xi,r_t \rangle = \langle \xi,h_0 \rangle + \int_0^t \big( \langle A^* \xi, r_s \rangle + \langle \xi,\alpha(r_s) \rangle \big) ds + \int_0^t \langle \xi,\sigma(r_{s-}) \rangle dX_s, \quad t \in \bbr_+,
\end{align*}
where we use the notation
\begin{align*}
\int_0^t \langle \xi,\sigma(r_{s-}) \rangle dX_s := \sum_{k=1}^m \int_0^t \langle \xi, \sigma^k(r_{s-}) \rangle dX_s^k, \quad t \in \bbr_+
\end{align*}
for the vector It\^{o} integral. We refer the reader, e.g., to \cite{P-Z-book} for further details.
\end{remark}

\begin{definition}
Let $B : \cald(B) \subset H \to H$ be the infinitesimal generator of a $C_0$-semigroup on $H$, and let $\beta : H \to H$ be a Lipschitz continuous mapping. Then the SPDEs (\ref{SPDE}) and (\ref{SPDE-B}) are called \emph{equivalent} if for each $h_0 \in H$ the weak solution to (\ref{SPDE}) with $r_0 = h_0$ coincides with the weak solution to (\ref{SPDE-B}) with $r_0 = h_0$.
\end{definition}

Let $V \subset H$ be a finite dimensional subspace. The following two auxiliary results show that the assumption from Theorem~\ref{thm-intro} that $V$ is $A$-semi-invariant does not mean a restriction.

\begin{lemma}\label{lemma-transformation}
There exists a linear operator $T \in L(H)$ such that $V$ is $B$-semi-invariant, where the linear operator $B : \cald(B) \subset H \to H$ is given by $\cald(B) := \cald(A)$ and $B := A + T$.
\end{lemma}

\begin{proof}
Let $H = U \oplus V$ be a direct sum decomposition of the Hilbert space $H$ with a closed subspace $U$. We denote by $\Pi_U : H \to U$ and $\Pi_V : H \to V$ the corresponding projections. There exists a subspace $E \subset V$ such that $V = (V \cap \cald(A)) \oplus E$. Let $\tilde{A} \in L(V,H)$ be the linear operator given by $\tilde{A}|_{V \cap \cald(A)} = A|_{V \cap \cald(A)}$ and $\tilde{A}|_E = 0$. We define the linear operator $T \in L(H)$ as $T := -\Pi_U \tilde{A} \Pi_V$. Then, for each $v \in V \cap \cald(A)$ we have
\begin{align*}
Bv = Av - \Pi_U A v = \Pi_V A v \in V,
\end{align*}
showing that $V$ is $B$-semi-invariant.
\end{proof}

\begin{lemma}\label{lemma-equiv}
Let $T \in L(H)$ be a linear operator, let the linear operator $B : \cald(B) \subset H \to H$ be given by $\cald(B) := \cald(A)$ and $B := A + T$, and let $\beta : H \to H$ be given by $\beta := \alpha - T$. Then the following statements are true:
\begin{enumerate}
\item $B$ is the generator of a $C_0$-semigroup on $H$.

\item $\beta$ is Lipschitz continuous.

\item The SPDEs (\ref{SPDE}) and (\ref{SPDE-B}) are equivalent. 
\end{enumerate} 
\end{lemma}

\begin{proof}
The first statement is a consequence of \cite[Thm.~3.1.1]{Pazy}, and the second statement follows from the Lipschitz continuity of $\alpha$ and $T$. For the proof of the third statement, let $h_0 \in H$ be arbitrary, and let $r$ be the weak solution to (\ref{SPDE-B}) with $r_0 = h_0$. Noting that $\cald(A^*) = \cald(B^*)$ and $B^* = A^* + T^*$, for each $\xi \in \cald(B^*)$ we obtain
\begin{align*}
\langle \xi,r_t \rangle &= \langle \xi,h_0 \rangle + \int_0^t \big( \langle B^* \xi, r_s \rangle + \langle \xi,\beta(r_s) \rangle \big) ds + \int_0^t \sigma(r_{s-}) dX_s
\\ &= \langle \xi,h_0 \rangle + \int_0^t \big( \langle A^* \xi + T^* \xi, r_s \rangle + \langle \xi,\alpha(r_s) - T r_s \rangle \big) ds + \int_0^t \sigma(r_{s-}) dX_s
\\ &= \langle \xi,h_0 \rangle + \int_0^t \big( \langle A^* \xi, r_s \rangle + \langle \xi,\alpha(r_s) \rangle \big) ds + \int_0^t \sigma(r_{s-}) dX_s, \quad t \in \bbr_+,
\end{align*}
showing that $r$ is also a weak solution to (\ref{SPDE}) with $r_0 = h_0$. An analogous calculation shows that the weak solution to (\ref{SPDE}) with $r_0 = h_0$ is also a weak solution to (\ref{SPDE-B}) with $r_0 = h_0$.
\end{proof}

\section{Direct sum decompositions of Hilbert spaces}\label{sec-Hilbert}

In this section, we will provide the required results about direct sum decompositions of Hilbert spaces. In particular, we will show that condition (2) from Theorem~\ref{thm-intro} does not depend on the choice of the decomposition. For what follows, let $H$ be a Hilbert space.

\begin{lemma}\label{lemma-const}
Let $V \subset H$ be a finite dimensional subspace, let $E \subset H$ be a subset, and let $\beta : E \to H$ be a mapping. Then the following statements are equivalent:
\begin{enumerate}
\item[(i)] There exists a closed subspace $U$ such that $H = U \oplus V$ and the mapping $\Pi_U \beta$ is constant on $E$.

\item[(ii)] For every closed subspace $U$ with $H = U \oplus V$ the mapping $\Pi_U \beta$ is constant on $E$.
\end{enumerate}
\end{lemma}

\begin{proof}
(i) $\Rightarrow$ (ii): Let $\tilde{U}$ be an arbitrary closed subspace such that $H = \tilde{U} \oplus V$. By assumption there exists $u \in U$ such that $\Pi_U \beta(h) = u$ for all $h \in E$. There exist unique $\tilde{u} \in \tilde{U}$ and $v \in V$ such that $u = \tilde{u} + v$. Therefore, we have
\begin{align*}
\beta(h) = \tilde{u} + v + \Pi_V \beta(h) \quad \text{for all $h \in E$,}
\end{align*}
and hence $\Pi_{\tilde{U}} \beta(h) = \tilde{u}$ for all $h \in E$, showing that $\Pi_{\tilde{U}} \beta$ is constant on $E$.

\noindent (ii) $\Rightarrow$ (i): This implication follows by choosing $U = V^{\perp}$.
\end{proof}

We use the following definition for the formulation of condition (2) from Theorem~\ref{thm-intro}.

\begin{definition}\label{def-const}
Let $V \subset H$ be a finite dimensional subspace, let $E \subset H$ be a subset, and let $\beta : E \to H$ be a mapping. We say that $\Pi_{(\bullet,V)} \beta$ is constant on $E$ if there exists a closed subspace $U$ such that $H = U \oplus V$ and the mapping $\Pi_U \beta$ is constant on $E$.
\end{definition}

\begin{remark}
By virtue of Lemma~\ref{lemma-const}, the Definition~\ref{def-const} does not depend on the choice of the subspace $U$.
\end{remark}

\section{Invariant foliations}\label{sec-foliations}

In this section, we will present the required results about $C^0$-foliations. The general mathematical framework is that of Section \ref{sec-SPDEs}. Let $V \subset H$ be a finite dimensional subspace. Throughout this section, we assume that $V$ is $A$-semi-invariant. Recall that, according to Lemmas~\ref{lemma-transformation} and  \ref{lemma-equiv}, this does not mean a restriction.

\begin{definition}
Let $k \in \bbn_0$ be a nonnegative integer. A family $(\calm_t)_{t \in \bbr_+}$ of subsets $\calm_t \subset H$, $t \in \bbr_+$ is called a \emph{$C^k$-foliation} generated by $V$ if there exists a mapping $\psi \in C^k(\bbr_+;H)$ such that
\begin{align*}
\calm_t = \psi(t) + V \quad \text{for all $t \in \bbr_+$.}
\end{align*}
In this case, the mapping $\psi$ is called a \emph{parametrization} of the foliation $(\calm_t)_{t \in \bbr_+}$.
\end{definition}

For what follows, let $(\calm_t)_{t \in \bbr_+}$ be a $C^0$-foliation generated by $V$. Here is the formal definition of invariance of the foliation.

\begin{definition}\label{def-inv-foliation}
The foliation $(\mathcal{M}_t)_{t \in \bbr_+}$ is called \emph{invariant} for the SPDE (\ref{SPDE}) if for all $t_0 \in
\mathbb{R}_+$ and $h_0 \in \mathcal{M}_{t_0}$ we have $r_{\bullet} \in \calm_{t_0 + \bullet}$ up to an evanescent set\footnote[1]{A random set $A \subset \Omega \times \mathbb{R}_+$ is called \emph{evanescent} if the set $\{ \omega \in \Omega : (\omega,t) \in A \text{ for some } t \in \mathbb{R}_+ \}$ is a $\mathbb{P}$-nullset, cf. \cite[1.1.10]{Jacod-Shiryaev}.}, where $r$ denotes the weak solution to (\ref{SPDE}) with $r_0 = h_0$.
\end{definition}

In order to prepare the notation for our next result, we define the union $\bbm := \bigcup_{t \in \bbr_+} \calm_t$. Furthermore, we fix a direct sum decomposition $H = U \oplus V$ of the Hilbert space $H$ with a closed subspace $U$, and denote by $\Pi_U : H \to U$ and $\Pi_V : H \to V$ the corresponding projections.

\begin{theorem}\label{thm-foliation}
The following statements are equivalent:
\begin{enumerate}
\item[(i)] The foliation $(\calm_t)_{t \in \bbr_+}$ is invariant for the SPDE (\ref{SPDE}).

\item[(ii)] The following conditions are satisfied:
\begin{align}\label{V-in-domain}
V &\text{ is $A$-invariant (or equivalently: $V \subset \cald(A)$)},
\\ \label{decomp-alpha} \Pi_U \alpha &\text{ is constant on $\calm_t$,} \quad \text{for each $t \in \bbr_+$,}
\\ \label{sigma-in-V} \sigma^k(\mathbb{M}) &\subset V, \quad k=1,\ldots,m,
\end{align}
and the weak solution $\psi : \bbr_+ \to H$ to the $H$-valued PDE 
\begin{align}\label{PDE}
\left\{
\begin{array}{rcl}
\frac{d\psi(t)}{dt} & = & A \psi(t) + \Pi_U \alpha(\psi(t))
\medskip
\\ \psi(0) & = & u_0,
\end{array}
\right.
\end{align}
where $u_0 \in U$ denotes the unique element such that $\mathcal{M}_0 \cap U = \{ u_0 \}$, is a parametrization of the foliation $(\calm_t)_{t \in \bbr_+}$.
\end{enumerate}
\end{theorem}

Before we provide the proof, we prepare an auxiliary result.

\begin{lemma}\label{lemma-foliations}
Suppose that conditions (\ref{V-in-domain})--(\ref{sigma-in-V}) are fulfilled, and let $\psi : \bbr_+ \to H$ be the weak solution to the PDE (\ref{PDE}). Then, for all $t_0 \in \bbr_+$ and all $v_0 \in V$ the following statements are true:
\begin{enumerate}
\item The SDE
\begin{align}\label{SDE-Y}
\left\{
\begin{array}{rcl}
dY_t & = & (A Y_t + \Pi_V \alpha(\psi(t_0+t) + Y_t)) dt + \sigma(\psi(t_0+t) + Y_{t-}) dX_t \medskip
\\ Y_0 & = & v_0
\end{array}
\right.
\end{align}
has a unique $V$-valued strong solution.

\item The process $r := \psi(t_0 + \bullet) + Y$ is the weak solution to the SPDE (\ref{SPDE}) with $r_0 = h_0$, where $h_0 := \psi(t_0) + v_0$.
\end{enumerate}
\end{lemma}

\begin{proof}
The first statement follows from (\ref{V-in-domain}) and (\ref{sigma-in-V}). For the proof of the second statement, let $\xi \in \cald(A^*)$ be arbitrary. Then, by (\ref{PDE}) and (\ref{decomp-alpha}) we have
\begin{align*}
&\langle \xi,\psi(t_0+t) \rangle = \langle \xi,\psi(t_0) \rangle + \langle \xi,\psi(t_0+t) - \psi(t_0) \rangle
\\ &= \langle \xi,\psi(t_0) \rangle + \int_{t_0}^{t_0+t} \big( \langle A^* \xi,\psi(s) \rangle + \langle \xi,\Pi_U \alpha (\psi(s)) \big) ds
\\ &= \langle \xi,\psi(t_0) \rangle + \int_{0}^{t} \big( \langle A^* \xi,\psi(t_0+s) \rangle + \langle \xi,\Pi_U \alpha(\psi(t_0+s) + Y_s) \big) ds, \quad t \in \bbr_+.
\end{align*}
Furthermore, by (\ref{SDE-Y}) we have
\begin{align*}
\langle \xi,Y_t \rangle &= \langle \xi,v_0 \rangle + \int_0^t \big( \langle A^* \xi,Y_s \rangle + \langle \xi,\Pi_V \alpha(\psi(t_0+s) + Y_s) \rangle \big) ds 
\\ &\quad + \int_0^t \langle \xi,\sigma(\psi(t_0+s) + Y_{s-}) \rangle dX_s, \quad t \in \bbr_+.
\end{align*}
Therefore, we arrive at
\begin{align*}
\langle \xi,r_t \rangle = \langle \xi,h_0 \rangle + \int_0^t \big( \langle A^* \xi,r_s \rangle + \langle \xi,\alpha(r_s) \rangle \big) ds + \int_0^t \langle \xi,\sigma(r_{s-}) \rangle dX_s, \quad t \in \bbr_+,
\end{align*}
showing that $r$ is the weak solution to (\ref{SPDE}) with $r_0 = h_0$.
\end{proof}

\begin{proof}[Proof of Theorem~\ref{thm-foliation}]
(i) $\Rightarrow$ (ii): Let $d := \dim V$ and let $\phi : \bbr_+ \to \bbm$ be a parametrization of the foliation $(\calm_t)_{t \in \bbr_+}$. According to \cite[Lemma~2.10]{Tappe-Wiener} there exist $\zeta_1,\ldots,\zeta_d \in \cald(A^*)$ and an isomorphism $T : \bbr^d \to V$ such that $T^{-1} = \langle \zeta,\bullet \rangle$, where we use the notation 
\begin{align*}
\langle \zeta,h \rangle := (\langle \zeta_1,h \rangle,\ldots,\langle \zeta_d,h \rangle) \in \bbr^d \quad \text{for $h \in V$.}
\end{align*}
We define the continuous mappings $\tilde{\alpha} : \bbr_+ \times \bbr^d \to \bbr^d$ and $\tilde{\sigma} : \bbr_+ \times \bbr^d \to (\bbr^d)^n$ as
\begin{align*}
\tilde{\alpha}(t,z) &:= \langle A^* \zeta, \phi(t) + Tz \rangle + \langle \zeta,\alpha(\phi(t)+Tz) \rangle,
\\ \tilde{\sigma}(t,z) &:= \langle \zeta,\sigma(\phi(t)+Tz) \rangle.
\end{align*}
Furthermore, for $t_0 \in \bbr_+$ and $h_0 \in \calm_{t_0}$ we define the process
\begin{align*}
Z^{h_0} := \langle \zeta,r^{h_0} - \phi(t_0 + \bullet) \rangle,
\end{align*}
where $r^{h_0}$ denotes the weak solution to (\ref{SPDE}) with $r_0 = h_0$. Now, let $t_0 \in \bbr_+$, $h_0 \in \calm_{t_0}$ and $v \in V$ be arbitrary. Then we have
\begin{equation}\label{proof-foli-1}
\begin{aligned}
&Z_t^{h_0 + v} - Z_t^{h_0} = \langle \zeta,r_t^{h_0 + v} - \phi(t_0 + t) \rangle - \langle \zeta,r_t^{h_0} - \phi(t_0 + t) \rangle = \langle \zeta, r_t^{h_0 + v} - r_t^{h_0} \rangle
\\ &= \langle \zeta,h_0+v \rangle + \int_0^t \big( \langle A^* \zeta, r_s^{h_0+v} \rangle + \langle \zeta, \alpha(r_s^{h_0+v}) \rangle \big) ds + \int_0^t \langle \zeta,\sigma(r_{s-}^{h_0+v}) \rangle dX_s
\\ &\quad - \langle \zeta,h_0 \rangle - \int_0^t \big( \langle A^* \zeta, r_s^{h_0} \rangle + \langle \zeta, \alpha(r_s^{h_0}) \rangle \big) ds - \int_0^t \langle \zeta,\sigma(r_{s-}^{h_0}) \rangle dX_s
\\ &= \langle \zeta,v \rangle + \int_0^t \big( \tilde{\alpha}(t_0 + s,Z_{s-}^{h_0 + v}) - \tilde{\alpha}(t_0 + s,Z_{s-}^{h_0}) \big) ds 
\\ &\quad + \int_0^t \big( \tilde{\sigma}(t_0 + s,Z_{s-}^{h_0 + v}) - \tilde{\sigma}(t_0 + s,Z_{s-}^{h_0}) \big) dX_s, \quad t \in \bbr_+.
\end{aligned}
\end{equation}
Furthermore, since the foliation $(\calm_t)_{t \in \bbr_+}$ is invariant for (\ref{SPDE}), we have $r_{\bullet}^{h_0}, r_{\bullet}^{h_0+v} \in \calm_{t_0 + \bullet}$ up to an evanescent set, and hence $r^{h_0 + v} - r^{h_0} \in V$ up to an evanescent set. Together with (\ref{proof-foli-1}) we obtain
\begin{align*}
r^{h_0 + v} - r^{h_0} &= T(Z^{h_0 + v} - Z^{h_0})
\\ &= v + \int_0^t T ( \tilde{\alpha}(t_0 + s,Z_{s-}^{h_0 + v}) - \tilde{\alpha}(t_0 + s,Z_{s-}^{h_0}) ) ds 
\\ &\quad + \int_0^t T ( \tilde{\sigma}(t_0 + s,Z_{s-}^{h_0 + v}) - \tilde{\sigma}(t_0 + s,Z_{s-}^{h_0}) ) dX_s, \quad t \in \bbr_+.
\end{align*}
Now, let $\xi \in \cald(A^*)$ be arbitrary. Then we have
\begin{equation}\label{proof-foli-2a}
\begin{aligned}
\langle \xi,r_t^{h_0 + v} - r_t^{h_0} \rangle &= \langle \xi,v \rangle + \int_0^t \langle \xi,T(\tilde{\alpha}(t_0 + s,Z_s^{h_0 + v}) - \tilde{\alpha}(t_0 + s,Z_{s-}^{h_0})) \rangle ds
\\ &\quad + \int_0^t \langle \xi,T(\tilde{\sigma}(t_0 + s,Z_{s-}^{h_0 + v}) - \tilde{\sigma}(t_0 + s,Z_{s-}^{h_0})) \rangle dX_s, \quad t \in \bbr_+.
\end{aligned}
\end{equation}
On the other hand, since $r^{h_0}$ and $r^{h_0 + v}$ are weak solutions to (\ref{SPDE}) with $r_0 = h_0$ and $r_0 = h_0 + v$, we have
\begin{equation}\label{proof-foli-2b}
\begin{aligned}
\langle \xi,r_t^{h_0 + v} - r_t^{h_0} \rangle &= \langle \xi,v \rangle + \int_0^t \big( \langle A^* \xi, r_s^{h_0 + v} - r_s^{h_0} \rangle + \langle \xi,\alpha(r_s^{h_0 + v}) - \alpha(r_s^{h_0}) \rangle \big) ds
\\ &\quad + \int_0^t \langle \xi, \sigma(r_{s-}^{h_0 + v}) - \sigma(r_{s-}^{h_0}) \rangle dX_s, \quad t \in \bbr_+.
\end{aligned}
\end{equation}
Combining (\ref{proof-foli-2a}) and (\ref{proof-foli-2b}), we obtain
\begin{align*}
\langle A^* \xi, v \rangle &= \big\langle \xi,T \big( \tilde{\alpha}(t_0 + s,\langle \zeta,h_0 - \phi(t_0) + v \rangle) - \tilde{\alpha}(t_0 + s,\langle \zeta,h_0 - \phi(t_0) \rangle) \big) \big\rangle 
\\ &\quad - \langle \xi,\alpha(h_0 + v) - \alpha(h_0) \rangle.
\end{align*}
This identity shows that $\xi \mapsto \langle A^* \xi, v \rangle$ is continuous on $\mathcal{D}(A^*)$, proving $v \in \mathcal{D}(A^{**})$. Since $A = A^{**}$, see \cite[Thm.~13.12]{Rudin}, we obtain $v \in \mathcal{D}(A)$, which yields (\ref{V-in-domain}). Therefore, we obtain
\begin{align*}
&\alpha(h_0 + v) - \alpha(h_0) 
\\ &= Av - T \big(\tilde{\alpha}(t_0 + s, \langle \zeta, h_0 - \phi(t_0) + v \rangle) - \tilde{\alpha}(t_0 + s, \langle \zeta, h_0 - \phi(t_0) \rangle ) \big) \in V,
\end{align*}
which shows that
\begin{align*}
\Pi_U \alpha(h_0 + v) - \Pi_U \alpha(h_0) = 0,
\end{align*}
proving (\ref{decomp-alpha}). A similar calculation as in (\ref{proof-foli-1}) and (\ref{proof-foli-2a}) shows that
\begin{equation}\label{proof-foli-3a}
\begin{aligned}
\langle \xi,r_t^{h_0} - \phi(t_0 + t) \rangle &= \langle \xi,h_0 - \phi(t_0 + t) \rangle + \int_0^t \langle \xi,T(\tilde{\alpha}(t_0 + s,Z_{s-}^{h_0})) \rangle ds
\\ &\quad + \int_0^t \langle \xi,T(\tilde{\sigma}(t_0 + s,Z_{s-}^{h_0})) \rangle dX_s, \quad t \in \bbr_+.
\end{aligned}
\end{equation}
On the other hand, since $r^{h_0}$ is a weak solution to (\ref{SPDE}), we have
\begin{equation}\label{proof-foli-3b}
\begin{aligned}
\langle \xi,r_t^{h_0} - \phi(t_0 + t) \rangle &= \langle \xi,h_0 - \phi(t_0 + t) \rangle + \int_0^t \big( \langle A^* \xi, r_s^{h_0} \rangle + \langle \xi,\alpha(r_s^{h_0}) \rangle \big) ds
\\ &\quad + \int_0^t \langle \xi, \sigma(r_{s-}^{h_0}) \rangle dX_s, \quad t \in \bbr_+.
\end{aligned}
\end{equation}
Therefore, we obtain
\begin{align*}
\sigma(h_0) = T \big( \tilde{\sigma}(t_0, \langle \zeta, h_0 - \phi(t_0) \rangle) \big) \in V^m,
\end{align*}
showing (\ref{sigma-in-V}). The remaining statement is a consequence of Lemma~\ref{lemma-foliations} (applied with $t_0 = 0$ and $v_0 = 0$) and the uniqueness of weak solutions to (\ref{SPDE}).

\noindent(ii) $\Rightarrow$ (i): This implication follows from Lemma~\ref{lemma-foliations} and the uniqueness of weak solutions to (\ref{SPDE}).
\end{proof}

\begin{remark}
Suppose that the mappings $\alpha : H \to H$ and $\sigma : H \to H^m$ are only continuous instead of being Lipschitz continuous. If we modify Definition~\ref{def-inv-foliation} by demanding the existence of an invariant solution for all $t_0 \in \bbr_+$ and $h_0 \in \calm_{t_0}$, then we can establish an analogous version of Theorem~\ref{thm-foliation}:
\begin{itemize}
\item The implication (i) $\Rightarrow$ (ii) remains true.

\item For the implication (ii) $\Rightarrow$ (i) we additionally assume the existence of weak solutions to (\ref{PDE}) and (\ref{SDE-Y}).
\end{itemize}
Consequently, analogous versions of Theorem~\ref{thm-intro} and its subsequent results also hold true without Lipschitz conditions -- provided that we have existence of weak solutions to equations of the types (\ref{PDE}) and (\ref{SDE-Y}).
\end{remark}

The invariance of $C^1$-foliations has been studied in \cite{Tappe-Wiener} and \cite{Tappe-Levy}. We recall that for a $C^1$-foliation $(\calm_t)_{t \in \bbr_+}$ and $t \in \bbr_+$ the tangent space is defined as $T \calm_t := \frac{d}{dt} \psi(t) + V$, where $\psi$ denotes a parametrization of $(\calm_t)_{t \in \bbr_+}$.

\begin{theorem}\label{thm-foliation-previous}
Suppose that $(\calm_t)_{t \in \bbr_+}$ is a $C^1$-foliation. Then the following statements are equivalent:
\begin{enumerate}
\item[(i)] The foliation $(\calm_t)_{t \in \bbr_+}$ is invariant for the SPDE (\ref{SPDE}).

\item[(ii)] We have
\begin{align}\label{M-in-domain}
\bbm &\subset \cald(A),
\\ \label{drift-tang} Ah + \alpha(h) &\in T \calm_t, \quad h \in \calm_t \text{ and } t \in \bbr_+,
\\ \sigma^k(\bbm) &\subset V, \quad k=1,\ldots,m.
\end{align}
\end{enumerate}
If the previous conditions are fulfilled, then for each $h_0 \in \bbm$ the weak solution to (\ref{SPDE}) with $r_0 = h_0$ is also a strong solution.
\end{theorem}

\begin{proof}
The proof is analogous to that of \cite[Thm.~2.11]{Tappe-Wiener}, and therefore omitted.
\end{proof}

\begin{remark}\label{remark-strong}
Suppose that the foliation $(\calm_t)_{t \in \bbr_+}$ is invariant for the SPDE (\ref{SPDE}). According to Theorems~\ref{thm-foliation} and \ref{thm-foliation-previous} the following statements are true:
\begin{itemize}
\item If $(\calm_t)_{t \in \bbr_+}$ is a $C^1$-foliation, then we have $\bbm \subset \cald(A)$, and for each $h_0 \in \bbm$ the weak solution to (\ref{SPDE}) with $r_0 = h_0$ is also a strong solution.

\item If $(\calm_t)_{t \in \bbr_+}$ is just a $C^0$-foliation, then we only have $V \subset \cald(A)$, and hence, for $h_0 \in \bbm$ the weak solution to (\ref{SPDE}) with $r_0 = h_0$ does not need to be a strong solution.
\end{itemize}
\end{remark}

The following result shows the relation between condition (\ref{decomp-alpha}) and the tangential condition (\ref{drift-tang}).

\begin{proposition}
Suppose we have (\ref{V-in-domain}) and that the PDE (\ref{PDE}) has a strong solution $\psi \in C^1(\bbr_+;H)$ with $\psi(\bbr_+) \subset \cald(A)$, which is a parametrization of the foliation $(\calm_t)_{t \in \bbr_+}$. Then the following statements are true:
\begin{enumerate}
\item We have (\ref{M-in-domain}).

\item Conditions (\ref{decomp-alpha}) and (\ref{drift-tang}) are equivalent.
\end{enumerate}
\end{proposition}

\begin{proof}
The first statement follows from (\ref{V-in-domain}) and the relation $\psi(\bbr_+) \subset \cald(A)$. For the proof of the second statement, let $t \in \bbr_+$ and $v \in V$ be arbitrary, and set $h := \psi(t) + v \in \calm_t$. By the PDE (\ref{PDE}) and condition (\ref{V-in-domain}) we obtain
\begin{align*}
Ah + \alpha(h) &= A \psi(t) + A v + \Pi_U \alpha(\psi(t) + v) + \Pi_V \alpha(\psi(t) + v)
\\ &= \frac{d}{dt} \psi(t) - \Pi_U \alpha(\psi(t)) + A v + \Pi_U \alpha(\psi(t) + v) + \Pi_V \alpha(\psi(t) + v)
\\ &= \underbrace{\frac{d}{dt} \psi(t) + A v + \Pi_V \alpha(\psi(t) + v)}_{\in T \calm_t} + \big( \Pi_U \alpha(\psi(t) + v) - \Pi_U \alpha(\psi(t)) \big),
\end{align*}
showing that conditions (\ref{decomp-alpha}) and (\ref{drift-tang}) are equivalent.
\end{proof}

In order to exemplify our previous results, consider the abstract Cauchy problem
\begin{align}\label{Cauchy}
\left\{
\begin{array}{rcl}
dr_t & = & A r_t dt 
\medskip
\\ r_0 & = & h_0.
\end{array}
\right.
\end{align}
Fix an arbitrary $h_0 \in H$ and let the foliation $(\calm_t)_{t \in \bbr_+}$ be given by $\calm_t := \{ S_t h_0 \}$. According to Theorem~\ref{thm-foliation}, the foliation $(\calm_t)_{t \in \bbr_+}$ is invariant for the abstract Cauchy problem (\ref{Cauchy}), and we can remark the following points:
\begin{itemize}
\item If $h_0 \in \cald(A)$, then $(\calm_t)_{t \in \bbr_+}$ is a $C^1$-foliation, and hence $\bbm \subset \cald(A)$.

\item If $A$ is the generator of a differentiable semigroup $(S_t)_{t \geq 0}$, then the mapping $t \mapsto S_t h_0$ is continuously differentiable on $(0,\infty)$ and we have $\calm_t \subset \cald(A)$ for all $t > 0$.
\end{itemize}
Finally, we present an example showing that the situation $\bbm \cap \cald(A) = \emptyset$ can occur. For this purpose, we choose the space of forward curves from \cite[Sec. 5]{fillnm}, which we will use later in Section \ref{sec-HJMM}. Let $H$ be the space of all absolutely continuous functions $h : \mathbb{R}_+ \rightarrow \mathbb{R}$ such that
\begin{align*}
\| h \| := \bigg( |h(0)|^2 + \int_{\mathbb{R}_+} |h'(x)|^2
w(x) dx \bigg)^{1/2} < \infty
\end{align*}
for some nondecreasing $C^1$-function $w : \bbr_+ \to [1,\infty)$ such that $w^{-1/3} \in \mathcal{L}^1(\bbr_+)$. Then the translation semigroup $(S_t)_{t \geq 0}$ is a $C_0$-semigroup on $H$ with generator $d/dx$ on the domain
\begin{align*}
\cald(d/dx) = \{ h \in C^1(\bbr_+) \cap H : h' \in H \}.
\end{align*}

\begin{example}\label{ex-HJM}
Let $h_0 : \bbr_+ \to \bbr$ be the unique absolutely continuous function with weak derivative
\begin{align*}
h_0' = \sum_{n \in \bbn_0} \mathbbm{1}_{[n,n+2^{-n} w(n)^{-1}]}.
\end{align*}
Then we have $h_0 \in H$, because $\| h_0 \| < \infty$, but for each $t \in \bbr_+$ we have $S_t h_0 \notin \cald(d/dx)$, because $S_t h_0 \notin C^1(\bbr_+)$, showing that $\bbm \cap \cald(d/dx) = \emptyset$. 
\end{example}

\section{Existence of affine realizations}\label{sec-realizations}

In this section, we provide the proofs of our main results concerning the existence of affine realizations. The general mathematical framework is that of Section \ref{sec-SPDEs}. We start with the formal definition of an affine realization.

\begin{definition}\mbox{}
\begin{enumerate}
\item Let $V \subset H$ be a finite dimensional subspace. We say that the SPDE (\ref{SPDE}) has an \emph{affine realization generated by $V$} if for all $h_0 \in H$ there is an invariant foliation $(\calm_t)_{t \in \bbr_+}$ generated by $V$ such that $h_0 \in \calm_0$.

\item We say that the SPDE (\ref{SPDE}) has an \emph{affine realization} if it has an affine realization generated by some finite dimensional subspace $V \subset H$.
\end{enumerate}
\end{definition}

With our preparations from Section \ref{sec-foliations}, we are now ready to provide the proofs of Theorems~\ref{thm-intro} and \ref{thm-linear}.

\begin{proof}[Proof of Theorem~\ref{thm-intro}]
If the SPDE (\ref{SPDE}) has an affine realization, then conditions (1)--(3) follow from Theorem~\ref{thm-foliation}.

Conversely, suppose that conditions (1)--(3) are fulfilled. Let $H = U \oplus V$ be a direct sum decomposition of the Hilbert space $H$ with a closed subspace $U$. Furthermore, let $h_0 \in H$ be arbitrary, and let $h_0 = u_0 + v_0$ be its decomposition according to $H = U \oplus V$.
Let $\psi$ be the weak solution to the PDE (\ref{PDE}), and let $(\calm_t)_{t \in \bbr_+}$ be the foliation $\mathcal{M}_t := \psi(t) + V$. Then we have $h_0 \in \calm_0$, and by Theorem~\ref{thm-foliation} the foliation $(\calm_t)_{t \in \bbr_+}$ is invariant for (\ref{SPDE-B}).
\end{proof}

\begin{proof}[Proof of Theorem~\ref{thm-linear}]
If conditions (1) and (2) are fulfilled, then, according to Theorem~\ref{thm-intro}, the linear SPDE (\ref{SPDE}) has an affine realization.

Conversely, suppose that the linear SPDE (\ref{SPDE}) has an affine realization. By Lemmas~\ref{lemma-transformation} and \ref{lemma-equiv} there exists a linear operator $T \in L(H)$ such that with the linear operator $B : \cald(B) \subset H \to H$ given by $\cald(B) := \cald(A)$ and $B := A + T$, and the mapping $\beta : H \to H$ given by $\beta := \alpha - T$, the following conditions are fulfilled:
\begin{itemize}
\item $V$ is $B$-semi-invariant.

\item $B$ is the generator of a $C_0$-semigroup on $H$.

\item $\beta$ is Lipschitz continuous.

\item The SPDEs (\ref{SPDE}) and (\ref{SPDE-B}) are equivalent. 
\end{itemize}
Let $H = U \oplus V$ be a direct sum decomposition of the Hilbert space $H$ with a closed subspace $U$. According to Theorem~\ref{thm-intro}, the subspace $V$ is $B$-invariant, we have that $\Pi_U \beta$ is constant on $V$, and we have $\sigma^k(H) \subset V$ for all $k=1,\ldots,m$. Noting that $\Pi_U \beta = \Pi_U \alpha - \Pi_U T$, and that $\alpha \in H$ is constant, we deduce that $\Pi_U T$ is constant on $V$, which implies $V \subset \ker(\Pi_U T)$. Therefore, the subspace $V$ is $T$-invariant, and hence it is also $A$-invariant.
\end{proof}

\begin{remark}\label{remark-construction}
Suppose that the SPDE (\ref{SPDE}) has an affine realization generated by some finite dimensional subspace $V$. 
\begin{itemize}
\item We can construct the curve $\psi$ and the $V$-valued process $Y$ appearing in (\ref{decomp-intro}) as follows. We fix a direct sum decomposition $H = U \oplus V$, and decompose an arbitrary starting point $h_0 \in H$ as $h_0 = u_0 + v_0$ according to $H = U \oplus V$. Inspecting the proofs of Theorems~\ref{thm-intro} and \ref{thm-foliation}, we see that $\psi : \bbr_+ \to H$ is the weak solution to the $H$-valued PDE (\ref{PDE}) and and that $Y$ is the strong solution to the $V$-valued SDE (\ref{SDE-Y}) with $t_0 = 0$. 

\item If $\alpha(H) \subset V$ (as in the situation of Corollary~\ref{cor-intro}), then the curve $\psi$ appearing in (\ref{decomp-intro}) is given by $\psi(t) = S_t h_0$ for $t \in \bbr_+$.

\item In any case, we can decompose the weak solution to the $H$-valued SPDE (\ref{SPDE}) into the weak solution to the $H$-valued PDE (\ref{PDE}) and the strong solution to the $V$-valued SDE (\ref{SDE-Y}). 

\item Even for $h_0 \in \cald(A)$ the invariant foliation is generally only a $C^0$-foliation, and hence, due to Remark~\ref{remark-strong}, the weak solution to the SPDE (\ref{SPDE}) is generally not a strong solution. 

\item If $\Pi_U \alpha(\cald(A)) \subset \cald(A)$ and $\Pi_U \alpha$ is Lipschitz continuous on $\cald(A)$ with respect to the graph norm
\begin{align*}
\| h \|_{\cald(A)} = \sqrt{ \| h \|^2 + \| Ah \|^2 }, \quad h \in \cald(A)
\end{align*}
(as, for example, in the situation of Corollary~\ref{cor-intro}), then, according to \cite[Thm.~6.1.7]{Pazy}, for each starting point $h_0 \in \cald(A)$ the PDE (\ref{PDE}) admits a classical solution, which implies that the invariant foliation is a $C^1$-foliation and that the weak solution to the SPDE (\ref{SPDE}) is also a strong solution.
\end{itemize}
\end{remark}

Finally, we will derive a result concerning the existence of affine realizations for linear SPDEs without specifying a finite dimensional subspace in advance. For this purpose, we require the concept of quasi-exponential volatilities.

\begin{definition}\label{def-qe}
We introduce the following notions:
\begin{enumerate}
\item If $\sigma^k(H) \subset \cald(A^{\infty})$ for all $k=1,\ldots,m$, then we define the subspace $A_{\sigma} \subset H$ as
\begin{align*}
A_{\sigma} := \sum_{k=1}^m \langle A^n \sigma^k(h) : n \in \bbn_0 \text{ and } h \in H \rangle.
\end{align*}
\item The volatility $\sigma$ is called \emph{$A$-quasi-exponential}, if we have $\sigma^k(H) \subset \cald(A^{\infty})$ for all $k=1,\ldots,m$ and $\dim A_{\sigma} < \infty$.
\end{enumerate}
\end{definition}

The following two auxiliary results are immediate consequences of Definition~\ref{def-qe}.

\begin{lemma}\label{lemma-sigma-qe}
Let $V$ be a finite dimensional $A$-invariant subspace such that $\sigma^k(H) \subset V$ for all $k=1,\ldots,m$. Then the volatility $\sigma$ is $A$-quasi-exponential.
\end{lemma}

\begin{lemma}\label{lemma-A-sigma-invariant}
Suppose that the volatility $\sigma$ is $A$-quasi-exponential, and set $V := A_{\sigma}$. Then $V$ is a finite dimensional $A$-invariant subspace, and we have $\sigma^k(H) \subset V$ for all $k=1,\ldots,m$.
\end{lemma}

Now, we are ready to formulate and prove the announced result.

\begin{theorem}\label{thm-qe}
Suppose that the SPDE (\ref{SPDE}) is linear. Then it has an affine realization if and only if the volatility $\sigma$ is $A$-quasi-exponential.
\end{theorem}

\begin{proof}
Suppose that the linear SPDE (\ref{SPDE}) has an affine realization. By Theorem~\ref{thm-linear} there exists a finite dimensional subspace $V \subset H$ such that $V$ is $A$-invariant and $\sigma^k(H) \subset V$ for all $k=1,\ldots,m$. According to Lemma~\ref{lemma-sigma-qe}, the volatility $\sigma$ is $A$-quasi-exponential.

Conversely, Suppose that the volatility $\sigma$ is $A$-quasi-exponential, and set $V := A_{\sigma}$. By Lemma~\ref{lemma-A-sigma-invariant}, the subspace $V$ is a finite dimensional $A$-invariant subspace, and we have $\sigma^k(H) \subset V$ for all $k=1,\ldots,m$. Therefore, by Theorem~\ref{thm-linear} the linear SPDE (\ref{SPDE}) has an affine realization.
\end{proof}

\section{The HJMM equation}\label{sec-HJMM}

In the section, we treat the HJMM equation as an example of a nonlinear SPDE. More precisely, we consider the SPDE
\begin{align}\label{HJMM-Wiener}
\left\{
\begin{array}{rcl}
dr_t & = & \big( \frac{d}{dx} r_t + \alpha_{\rm HJM}(r_t) \big) dt +
\sigma(r_{t})dW_t
\medskip
\\ r_0 & = & h_0
\end{array}
\right.
\end{align}
driven by a $\bbr^m$-valued Wiener processes $W$. The state space $H$ of (\ref{HJMM-Wiener}) is the space used in Example \ref{ex-HJM}. The weak solutions $r$ to (\ref{HJMM-Wiener}) are interest rate curves in a market of zero coupon bonds. In order to ensure that this bond market is free of arbitrage, we assume that the drift term in (\ref{HJMM-Wiener}) is given by the HJM drift condition
\begin{align}\label{HJM-drift-Wiener}
\alpha_{\rm HJM}(h) = \sum_{k=1}^m \sigma^k(h) \cdot T \sigma^k(h),
\end{align}
where $T : H \to H$ denotes the integral operator given by $T h := \int_0^{\bullet} h(\eta) d\eta$ for $h \in H$. We refer, e.g., to \cite{fillnm} for further details concerning the derivation of the HJMM equation (\ref{HJMM-Wiener}) and the HJM drift condition (\ref{HJM-drift-Wiener}).

Our goal of this section is to provide an alternative and rather short proof of a well-known result concerning the existence of FDRs for the HJMM equation (\ref{HJMM-Wiener}), which can, e.g., be found in \cite{Bj_Sv}, \cite{Bj_La} or \cite{Tappe-Wiener}. For this purpose, we start with an auxiliary result.

\begin{lemma}\label{lemma-product}
Let $V$ be a finite dimensional $(d/dx)$-invariant subspace. Then the subspace $V + P(V)$, where
\begin{align}\label{product-space}
P(V) := \langle h \cdot g : h \in V \text{ and } g \in T V \rangle,
\end{align}
is finite dimensional and $(d/dx)$-invariant, too.
\end{lemma}

\begin{proof}
The subspace $V + P(V)$ is finite dimensional, because we have
\begin{align*}
\dim (V + P(V)) \leq \dim V + (\dim V)^2 < \infty.
\end{align*}
Let $h \in V$ and $g \in T V$ be arbitrary. Then there exists $f \in V$ such that $g = Tf$. Since $V$ is $(d/dx)$-invariant, we obtain $f' \in V$, and hence
\begin{align*}
\frac{d}{dx} g = \frac{d}{dx} \int_0^{\bullet} f(y) dy = f = \big( f - f(0) \big) + f(0) = \int_0^{\bullet} f'(y) dy + f(0) \in T V + \langle 1 \rangle.
\end{align*}
Therefore, and since $V$ is $(d/dx)$-invariant, we deduce
\begin{align*}
\frac{d}{dx} \big( h \cdot g \big) = \frac{d}{dx} h \cdot g + h \cdot \frac{d}{dx} g \in V + P(V),
\end{align*}
showing that $V + P(V)$ is $(d/dx)$-invariant.
\end{proof}

\begin{proposition}\label{prop-HJMM-Wiener}
Suppose that the volatility $\sigma$ is $(d/dx)$-quasi-exponential. Then the HJMM equation (\ref{HJMM-Wiener}) has an affine realization.
\end{proposition}

\begin{proof}
For simplicity of notation, we set $A := d/dx$.
By Lemma~\ref{lemma-A-sigma-invariant}, the subspace $A_{\sigma}$ is a finite dimensional $A$-invariant subspace, and we have $\sigma^k(H) \subset A_{\sigma}$ for all $k=1,\ldots,m$. By Lemma~\ref{lemma-product}, the subspace $V := A_{\sigma} + P(A_{\sigma})$ is finite dimensional and $A$-invariant, too. Moreover, we have $\sigma^k(H) \subset A_{\sigma} \subset V$ for all $k=1,\ldots,m$, and by (\ref{HJM-drift-Wiener}) and (\ref{product-space}) we have $\alpha_{\rm HJM}(H) \subset V$. Therefore, Corollary~\ref{cor-intro} concludes the proof.
\end{proof}

\begin{remark}
Suppose that the volatility $\sigma$ is $(d/dx)$-quasi-exponential.
\begin{itemize}
\item Note that the just presented result is more general than \cite[Prop.~6.2]{Tappe-Wiener}, because here we obtain a representation of the form (\ref{decomp-intro}) for every starting point $h_0 \in H$, whereas the aforementioned result only provides such a representation for each starting point $h_0 \in \cald(d/dx)$.

\item For each $h_0 \in H$ the curve $\psi$ appearing in (\ref{decomp-intro}) is given by $\psi(t) = S_t h_0$ for $t \in \bbr_+$, which follows from Remark~\ref{remark-construction}. Furthermore, for each $h_0 \in \cald(d/dx)$ the invariant foliation is a $C^1$-foliation and the weak solution to the HJMM equation (\ref{HJMM-Wiener}) is also a strong solution.

\item If we add driving L\'{e}vy processes with jumps in the HJMM equation (\ref{HJMM-Wiener}), then the statement of Proposition~\ref{prop-HJMM-Wiener} is no longer true, because the drift condition becomes more involved. We refer to \cite{Tappe-Levy} for details on this subject.
\end{itemize}
\end{remark}

\section{Examples of linear SPDEs}\label{sec-examples}

In this section, we present several examples of linear SPDEs arising in natural sciences and economics. Our approach in these example is to determine all finite dimensional invariant subspaces, and to apply Theorem~\ref{thm-linear} afterwards. For this procedure, we determine all eigenvalues $\lambda$ of the generator $A$, and then we distinguish two cases:
\begin{itemize}
\item For a general operator $A$, we determine all solutions of the generalized eigenvalue problem. More precisely, let $\lambda \in \bbc$ be an eigenvalue of $A$ and let $n \in \bbn$ be arbitrary. If $\lambda \in \bbr$, then we determine all solutions of the generalized eigenvalue problem
\begin{align}\label{EV-1a}
(A - \lambda)^n = 0,
\end{align}
and in the case $\lambda \in \bbc \setminus \bbr$ we determine all solutions of the generalized eigenvalue problem
\begin{align}\label{EV-1b}
((A - \lambda)(A - \overline{\lambda}))^n = 0.
\end{align}
\item If $A$ is symmetric\footnote[2]{For our purposes, we do not need that the operator $A$ is self-adjoint, because we merely consider its restrictions on finite dimensional subspaces of $H$.}, then every eigenvalue is real, and for an eigenvalue $\lambda \in \bbr$ it suffices to determine all solutions of the eigenvalue problem
\begin{align}\label{EV-2}
A - \lambda = 0.
\end{align}
\end{itemize}
Our general mathematical framework in this section is that of Section~\ref{sec-SPDEs}; in particular, throughout this section, the driving process $X$ denotes a $\bbr^m$-valued L\'{e}vy process for some positive integer $m \in \bbn$.

First, we deal with the HJMM equation, which we have already encountered in Section \ref{sec-HJMM}. Here we consider the linear HJMM equation
\begin{align}\label{HJMM}
\left\{
\begin{array}{rcl}
dr_t & = & \big( \frac{d}{dx} r_t + \alpha_{\rm HJM} \big) dt +
\sigma dX_t
\medskip
\\ r_0 & = & h_0.
\end{array}
\right.
\end{align}
In order to be consistent with the upcoming examples, we consider (\ref{HJMM}) on the state space $L^2(\bbr_+,\rho)$ for some appropriate measure $\rho$. Moreover, in order to ensure the absence of arbitrage, we assume that the drift term is given by
\begin{align*}
\alpha_{\rm HJM} = \frac{d}{dx} \Psi(- T \sigma),
\end{align*}
where $\Psi$ denotes the cumulant generating function of the L\'{e}vy process $X$. We refer, e.g., to \cite[Sec.~2.1]{Eberlein_O} for further details.

\begin{proposition}\label{prop-HJMM-linear}
The following statements are equivalent:
\begin{enumerate}
\item[(i)] The linear HJMM equation (\ref{HJMM}) has an affine realization.

\item[(ii)] There are finite sets $I \subset \bbr$, $J \subset \bbr \times (0,\infty)$, and an integer $p \in \bbn_0$ such that
\begin{equation}\label{sin-cos}
\begin{aligned}
\sigma^k &\in \bigoplus_{\lambda \in I} \langle x \mapsto x^j \exp(\lambda x) : j = 0,\ldots,p \rangle 
\\ &\quad \oplus \bigoplus_{(\mu,\nu) \in J} \langle x \mapsto x^j \exp(\mu x) \cos(\nu x),
\\ &\qquad\qquad\quad\,\,\, x \mapsto x^j \exp(\mu x) \sin(\nu x) : j = 0,\ldots,p \rangle
\end{aligned}
\end{equation}
for all $k=1,\ldots,m$.
\end{enumerate}
\end{proposition}

\begin{proof}
We set $A := d/dx$, and let $n \in \bbn$ be arbitrary. For $\lambda \in \bbr$ all solutions to the ODE (\ref{EV-1a})
are given by the linear space
\begin{align*}
\langle x \mapsto x^j \exp(\lambda x) : j = 0,\ldots,n-1 \rangle.
\end{align*}
Furthermore, for $\lambda = \mu + i \nu \in \bbc \setminus \bbr$ with $\nu > 0$ all solutions to the ODE (\ref{EV-1b}) are given by the linear space
\begin{align*}
\langle x \mapsto x^j \exp(\mu x) \cos(\nu x), x \mapsto x^j \exp(\mu x) \sin(\nu x) : j = 0,\ldots,n-1 \rangle.
\end{align*}
Therefore, applying Theorem~\ref{thm-linear} completes the proof.
\end{proof}

\begin{remark}
We refer to \cite[Thm.~5]{Zabczyk} for a closely related result regarding the linear HJMM equation driven by Wiener processes.
\end{remark}

Next, we consider the stochastic transport equation
\begin{align}\label{SPDE-transport}
\left\{
\begin{array}{rcl}
du_t & = & \big( \langle v,\nabla \rangle u_t + \alpha \big) dt +
\sigma(u_{t-}) dX_t
\medskip
\\ u_0 & = & h_0,
\end{array}
\right.
\end{align}
which describes the contaminant of a fluid with velocity $v \in \bbr^d$ over time. Here the state space is $H = L^2(C,\rho)$ with a closed set $C \subset \bbr^d$ and an appropriate measure $\rho$. We assume that the closed set $C$ has the property
\begin{align*}
C = \partial C + \{ t v : t \in \bbr_+ \},
\end{align*}
and that for every $y \in C$ there exist unique elements $x \in \partial C$ and $t \in \bbr_+$ such that $y = x + tv$.
The first order differential operator $\langle v,\nabla \rangle$ appearing in (\ref{SPDE-transport}) is generated by the translation semigroup $(S_t u)(x) = u(x + tv)$ for $t \geq 0$ and $x \in C$. Here are two examples which are covered by this framework:
\begin{itemize}
\item The HJMM equation (\ref{HJMM}), where we have $C = \bbr_+$, $\partial C = \{ 0 \}$ and $v = 1$.

\item The SPDE presented in \cite{Tappe-Weber}, which describes the mortality rates of demographic evolutions. Here the sets $C, \partial C \subset \bbr^2$ are given by
\begin{align*}
C &= \{ (s,y) \in \bbr_+ \times \bbr : y \geq -s \},
\\ \partial C &= \{ t (0,1) : t \in \bbr_+ \} \cup \{ t (1,-1) : t \in \bbr_+ \},
\end{align*}
and we have the velocity $v = (1,-1)$.
\end{itemize}

\begin{proposition}
We suppose there exist functions $\xi : \partial C \to \bbr^m$ and $h : \bbr_+ \to \bbr^m$ such that
\begin{align*}
\sigma^k(x + tv) = \xi^k(x) \cdot h^k(t) \quad \text{for all $(x,t) \in \partial C \times \bbr_+$ and all $k=1,\ldots,m$,}
\end{align*}
and $h^k$ is of the form (\ref{sin-cos}) for all $k = 1,\ldots,m$. Then the stochastic transport equation (\ref{SPDE-transport}) has an affine realization.
\end{proposition}

\begin{proof}
Setting $A := \langle v,\nabla \rangle$, for all $x \in \partial C$ and all $t \in \bbr_+$ we have
\begin{align*}
A \sigma^k(x + tv) = \xi^k(x) \cdot (h^k)'(t), \quad k = 1,\ldots,m,
\end{align*}
and hence, combining Theorem~\ref{thm-linear} and Proposition~\ref{prop-HJMM-linear} concludes the proof.
\end{proof}

Now, we consider examples of second order operators, with corresponding applications typically arising in natural sciences. Our first such example is the stochastic cable equation (cf. \cite[Ex. 0.8]{Da_Prato})
\begin{align}\label{SPDE-cable}
\left\{
\begin{array}{rcl}
dv_t & = & \frac{1}{\tau} \big( \lambda^2 \frac{d^2}{d x^2} v_t - v_t \big) dt +
\sigma dX_t
\medskip
\\ v_0 & = & h_0,
\end{array}
\right.
\end{align}
which describes the voltage of an electric cable over time. The constants $\lambda,\tau > 0$ are physical constants of the electric cable; $\lambda$ is the length constant and $\tau$ is the time constant. Here the state space is $H = L^2((0,\pi))$ and we can choose the generator $A = - \frac{d^2}{d x^2}$ on the domain $\cald(A) = H^2((0,\pi)) \cap H_0^1((0,\pi))$. Thus, the electric cable is modeled by the interval $[0,\pi]$ and we consider Dirichlet boundary conditions, which means that there is no voltage at the end points of the cable.

\begin{proposition}
The following statements are equivalent:
\begin{enumerate}
\item[(i)] The stochastic cable equation (\ref{SPDE-cable}) has an affine realization.

\item[(ii)] There is a finite index set $I \subset \bbn$ such that
\begin{align*}
\sigma^k &\in \bigoplus_{n \in I} \langle x \mapsto \sin(n x) \rangle, \quad k = 1,\ldots,m.
\end{align*}
\end{enumerate}
\end{proposition}

\begin{proof}
The eigenvalues of the Sturm-Liouville eigenvalue problem
\begin{align*}
u'' + \lambda u = 0, \quad u(0) = u(\pi) = 0
\end{align*}
are given by $\lambda_n = n^2$, $n \in \bbn$, and the corresponding eigenfunctions are given by
\begin{align*}
u_n(x) = \sin(nx), \quad n \in \bbn.
\end{align*}
Therefore, Theorem~\ref{thm-linear} completes the proof.
\end{proof}

Next, we consider the stochastic heat equation
\begin{align}\label{SPDE-heat}
\left\{
\begin{array}{rcl}
du_t & = & a \Delta u_t dt + \sigma dX_t
\medskip
\\ u_0 & = & h_0,
\end{array}
\right.
\end{align}
which describes the heat of a medium in a region over time. The constant $a > 0$ is the heat conductivity. Here we have the state space $H = L^2(O)$, where $O \subset \bbr^2$ denotes the open unit ball
\begin{align*}
O = \{ x \in \bbr^2 : x_1^2 + x_2^2 < 1 \},
\end{align*}
and we can choose the generator $A = -\Delta$ on the domain $\cald(A) = H^2(O) \cap H_0^1(O)$. Therefore, the region, in which we measure the temperature, is the closed unit ball $\overline{O}$ and we consider Dirichlet boundary conditions, which means that the temperature is zero at the boundary $\partial O$ of the ball. In the upcoming result, we use polar coordinates, and we agree on the following notation:
\begin{itemize}
\item For $p \in \bbn_0$ we denote by $J_p : \bbr_+ \to \bbr$ the Bessel function of the first kind.

\item For $(p,q) \in \bbn_0 \times \bbn$ we denote by $\lambda_{pq} > 0$ the $q$-th positive zero of the Bessel function $J_p$.
\end{itemize}

\begin{proposition}
The following statements are equivalent:
\begin{enumerate}
\item[(i)] The stochastic heat equation (\ref{SPDE-heat}) has an affine realization.

\item[(ii)] There is a finite index set $I \subset \bbn_0 \times \bbn$ such that
\begin{align*}
\sigma^k &\in \bigoplus_{(p,q) \in I} \langle (r,\varphi) \mapsto \cos(p \varphi) J_p(\lambda_{pq} r), (r,\varphi) \mapsto \sin(p \varphi) J_p(\lambda_{pq} r) \rangle
\end{align*}
for all $k = 1,\ldots,m$. 
\end{enumerate}
\end{proposition}

\begin{proof}
The eigenvalues of the Laplace eigenvalue problem
\begin{align*}
\Delta u + \lambda u = 0, \quad u = 0 \text{ on $\partial O$}
\end{align*}
are given by $\lambda_{pq}^2$, $(p,q) \in \bbn_0 \times \bbn$, and the corresponding eigenfunctions are, by using polar coordinates, given by
\begin{align*}
u_{pq}(r,\varphi) = \cos(p \varphi) J_p(\lambda_{pq} r) \quad \text{and} \quad v_{pq}(r,\varphi) = \sin(p \varphi) J_p(\lambda_{pq} r).
\end{align*}
Therefore, Theorem~\ref{thm-linear} completes the proof.
\end{proof}

Both, the Hermite semigroup (also called Dunkl-Hermite semigroup or Ornstein-Uhlenbeck semigroup) and the Laguerre semigroup play a central role in quantum mechanics and mathematical physics. First, we consider the stochastic Hermite equation
\begin{align}\label{SPDE-Hermite}
\left\{
\begin{array}{rcl}
du_t & = & \big( - \frac{\Delta}{2} + \langle x,\nabla \rangle \big) u_t dt + \sigma(u_{t-}) dX_t
\medskip
\\ u_0 & = & h_0
\end{array}
\right.
\end{align}
on the state space $H = L^2(\bbr^d,\exp(- \| x \|_2^2)dx)$ for some $d \in \bbn$. If $d \geq 2$, then for $\beta \in \bbn_0^d$ we define the generalized Hermite polynomial $H_{\beta}$ as
\begin{align*}
H_{\beta}(x) := \prod_{i=1}^d H_{\beta_i}(x_i), \quad x \in \bbr^d,
\end{align*}
where the $(H_n)_{n \in \bbn_0}$ denote the usual Hermite polynomials.

\begin{proposition}
The following statements are equivalent:
\begin{enumerate}
\item The stochastic Hermite equation (\ref{SPDE-Hermite}) has an affine realization.

\item There is a finite index set $I \subset \bbn_0 $ such that
\begin{align*}
\sigma^k(H) &\subset \bigoplus_{n \in I} \langle H_{\beta} : \beta \in \bbn_0^d \text{ with } |\beta| = n \rangle
\end{align*}
for all $k = 1,\ldots,m$. 
\end{enumerate}
\end{proposition}

\begin{proof}
The eigenvalue problem
\begin{align*}
- \frac{\Delta u}{2} + \langle x,\nabla u \rangle = \lambda u
\end{align*}
has the eigenvalues $\lambda_n = n$, $n \in \bbn_0$ with corresponding eigenfunctions
\begin{align*}
\{ H_{\beta} : \beta \in \bbn_0^d \text{ with } |\beta| = n \}.
\end{align*}
Therefore, Theorem~\ref{thm-linear} completes the proof.
\end{proof}

Next, we consider the stochastic Laguerre equation
\begin{align}\label{SPDE-Laguerre}
\left\{
\begin{array}{rcl}
du_t & = & - \big( \langle x,\partial^2 \rangle + \langle 1 - x, \nabla \rangle \big) u_t dt + \sigma(u_{t-}) dX_t
\medskip
\\ u_0 & = & h_0
\end{array}
\right.
\end{align}
on the state space $H = L^2(\bbr_+^d,\calb(\bbr_+^d),\exp(-\|x\|_1))$ for some $d \in \bbn$. If $d \geq 2$, then for $\beta \in \bbn_0^d$ we define the generalized Laguerre polynomial $L_{\beta}$ as
\begin{align*}
L_{\beta}(x) := \prod_{i=1}^d L_{\beta_i}(x_i), \quad x \in \bbr^d,
\end{align*}
where the $(L_n)_{n \in \bbn_0}$ denote the usual Laguerre polynomials.

\begin{proposition}
The following statements are equivalent:
\begin{enumerate}
\item The stochastic Laguerre equation (\ref{SPDE-Laguerre}) has an affine realization.

\item There is a finite index set $I \subset \bbn_0 $ such that
\begin{align*}
\sigma^k(H) &\subset \bigoplus_{n \in I} \langle L_{\beta} : \beta \in \bbn_0^d \text{ with } |\beta| = n \rangle
\end{align*}
for all $k = 1,\ldots,m$. 
\end{enumerate}
\end{proposition}

\begin{proof}
The eigenvalue problem
\begin{align*}
\langle x, \partial^2 u \rangle + \langle 1-x, \nabla u \rangle + \lambda u = 0
\end{align*}
has the eigenvalues $\lambda_n = n$, $n \in \bbn_0$ with corresponding eigenfunctions
\begin{align*}
\{ L_{\beta} : \beta \in \bbn_0^d \text{ with } |\beta| = n \}.
\end{align*}
Therefore, Theorem~\ref{thm-linear} completes the proof.
\end{proof}

In \cite{Cont}, a model for the term structure of interest rates, which is different from the HJMM equation (\ref{HJMM}), was proposed. Namely, it was assumed that the fluctuation process satisfies a second order SPDE of the form
\begin{align}\label{SPDE-bond-2}
\left\{
\begin{array}{rcl}
dY_t & = & \big( \frac{\kappa}{2} \frac{d^2}{d x^2} Y_t + \frac{d}{d x}Y_t \big) dt + \sigma dX_t
\medskip
\\ Y_0 & = & h_0
\end{array}
\right.
\end{align}
with a positive constant $\kappa > 0$ and Dirichlet boundary conditions. Here the state space is $H = L^2((0,1),\exp ( x / \kappa )dx)$, and we can choose the generator
\begin{align*}
A = - \frac{\kappa}{2} \frac{d^2}{d x^2} - \frac{d}{d x}
\end{align*}
on the domain $\cald(A) = H^2((0,1)) \cap H_0^1((0,1))$.

\begin{proposition}
The following statements are equivalent:
\begin{enumerate}
\item The second order term structure equation (\ref{SPDE-bond-2}) has an affine realization.

\item There is a finite index set $I \subset \bbn$ such that
\begin{align*}
\sigma^k &\in \bigoplus_{n \in I} \langle x \mapsto \exp(-x/\kappa) \sin(n \pi x) \rangle
\end{align*}
for all $k = 1,\ldots,m$. 
\end{enumerate}
\end{proposition}

\begin{proof}
The eigenvalue problem
\begin{align*}
\frac{\kappa}{2} u'' + u' + \lambda u = 0, \quad u(0) = u(1) = 0
\end{align*}
has the eigenvalues
\begin{align*}
\lambda_n = \frac{1}{2 \kappa} \big( 1 + n^2 \pi^2 \kappa^2 \big), \quad n \in \bbn,
\end{align*}
with corresponding eigenfunctions
\begin{align*}
u_n(x) = \exp ( - x / \kappa ) \sin(n \pi x), \quad n \in \bbn.
\end{align*}
Therefore, Theorem~\ref{thm-linear} completes the proof.
\end{proof}

\begin{remark}
We refer to \cite[Thm.~6]{Zabczyk} for a closely related result regarding the second order term structure equation (\ref{SPDE-bond-2}) driven by Wiener processes.
\end{remark}

\end{document}